\documentclass[11pt]{amsart}

\usepackage{amssymb,amsfonts,latexsym}
\usepackage[centertags]{amsmath}
\usepackage{amsfonts}
\usepackage{amssymb}
\usepackage{amsthm}
\usepackage[all]{xy}
\usepackage{cite}
\usepackage{graphicx,color}
\usepackage{todo}
\xyoption{poly}
\newtheorem{theorem}{Theorem}[section]
\newtheorem{lemma}[theorem]{Lemma}
\newtheorem{proposition}[theorem]{Proposition}
\newtheorem{corollary}[theorem]{Corollary}

\theoremstyle{definition}
\newtheorem{definition}[theorem]{Definition} % \theoremstyle{remark}
\newtheorem{remark}[theorem]{Remark}
\newtheorem{problem}[theorem]{Problem}
\newtheorem{example}[theorem]{Example}

\newcommand{\Pic}{\operatorname{Pic}}

\newcommand{\cE}{{\mathcal E}}

\newcommand{\cF}{{\mathcal F}}

\newcommand{\cO}{{\mathcal O}}

\newcommand {\ZZ}{\mathbb{Z}}
\newcommand {\NN}{\mathbb{N}}
\newcommand {\PP}{\mathbb{P}}
\newcommand{\cS}{{\mathcal S}}
\newcommand{\cQ}{{\mathcal Q}}

\DeclareMathOperator{\GL}{ GL}
\DeclareMathOperator{\Gr}{Gr}
\newcommand{\odi}[1]{\mathcal{O}_{#1}}
\newcommand{\shE}{\mathcal{E}}
\DeclareMathOperator{\rk}{rk}
\DeclareMathOperator{\Hl}{H}
\DeclareMathOperator{\h}{h}
\DeclareMathOperator{\depth}{depth}
\DeclareMathOperator{\de}{deg}

\begin{document}
\title[$\GL(V)$-invariant Ulrich bundles on Grassmannians]
{$\mathbf{GL(V)}$-invariant Ulrich bundles on Grassmannians}

\author[L.\ Costa, R.M.\ Mir\'o-Roig]{L.\ Costa$^*$, R.M.\
Mir\'o-Roig$^{**}$}

\address{Facultat de Matem\`atiques,
Departament d'Algebra i Geometria, Gran Via de les Corts Catalanes
585, 08007 Barcelona, SPAIN } \email{costa@ub.edu}

\address{Facultat de Matem\`atiques,
Departament d'Algebra i Geometria, Gran Via de les Corts Catalanes
585, 08007 Barcelona, SPAIN } \email{miro@ub.edu}

\date{\today}
\thanks{$^*$ Partially supported by MTM2010-15256.}
\thanks{$^{**}$ Partially supported by MTM2010-15256.}


%%%%%%%%%%%%%%%%%%%%%%%%%%%%%%%%

\begin{abstract}
In this paper, we give a full classification of all homogeneous Ulrich bundles on a Grassmannian $\Gr(k,n)$ of
$k$-planes on $\PP^n$.
\end{abstract}

%%%%%%%%%%%%%%%%%%%%%%%%%%%%%%%

\maketitle

\tableofcontents

%%%%%%%%%%%%%%%%%%%%%%%%%%%%%%%%%%%%%%%%%%%%%%%%%
\section{Introduction}

 The existence of Ulrich bundles (i.e. bundles without intermediate cohomology  whose corresponding module has the maximal number of generators) on a projective variety is a challenging problem with a long and interesting history behind and few known examples. For example, Brennan, Herzog,  and Ulrich   proved the existence of Ulrich bundles on a complete intersection
 and on a linear determinantal variety (\cite{BHU}, Proposition 2.8); Buchsweitz, Eisenbud and Herzog  showed that the minimal rank of
 an Ulrich  bundle on a hyperquadric $Q\subset \PP^n$ is $2^{\lfloor \frac{n}{2} \rfloor -1}$;  and the $d$-uple
 embedding of $\PP^n$ has an Ulrich bundle if $n\le 2$ or $n=3$ and $3\nmid d$ (\cite{H}, Theorem 3.7  and \cite{ESW},
 Proposition 5.10).  See also  \cite{CH}, \cite{CMRP}, \cite{MR}, \cite{MRP} and  \cite{MRP1} and the references in those papers for
 other examples of Ulrich bundles.

 The notion of Ulrich bundle has its origins on its algebraic counterpart. In the eighties, Ulrich in  \cite{Ulr} (see also \cite{BHU}) gave an upper bound for the minimal number of generators of a Maximal Cohen-Macaulay module over a Cohen-Macaulay homogeneous ring.  Modules attaining the bound are called  Ulrich modules
 and correspondingly Ulrich bundles in the geometric set up.  A detailed account on Ulrich bundles is provided in
 \cite{ESW}. In particular, we know that Ulrich bundles on a projective variety $X\subset \PP^n$ admit a linear
 locally free $\odi{\PP^n}$-resolution (\cite{ESW}, Proposition 2.1) and for that reason they have been
 also studied under the name linear maximal Cohen-Macaulay modules (see, for instance, \cite{H}).

 \vspace{3mm}

 In \cite{ESW}, pg. 543,  Eisenbud, Schreyer and Weyman leave open  the following  problem

\begin{problem} \label{problem} (a) Is every  variety (or even scheme) $X\subset \PP^n$ the support of an Ulrich
sheaf?

(b) If so, what is the smallest possible rank  for such a sheaf?
\end{problem}

Recently, after the Boij-S\"{o}derberg theory has been developed, the interest on these questions has grown up due to the fact that it has been proved (\cite{ES2}, Theorem 4.2) that the existence of an Ulrich sheaf on a smooth projective variety $X$ of dimension $n$ implies that the cone of cohomology tables of vector bundles on $X$ is the same as the one of vector bundles on $\PP^n$.

\vspace{3mm}

The goal of this work is to solve both Problems  for the Grassmann variety $\Gr(k,n)$
of $k$-dimensional  linear subspaces of $\PP^n$. Even more, we will explicitly determine all $\GL(V)$-irreducible invariant
Ulrich bundles on $\Gr(k,n)$.  As a main tool we will use Schur
functors and the Borel-Bott-Weil theorem which computes  the cohomology of any
irreducible $\GL(V)$-invariant vector bundle $\Sigma^{\alpha}\cQ \otimes \Sigma^{\gamma}\cS^{\vee}$ on $\Gr(k,n)$.
This fact  is what limits  our result to characteristic zero.

\vskip 4mm
Next we outline the structure. The paper is organized in 3 sections: in section 2, we fix the notation, we introduce
the definition  and basic properties of  Ulrich bundles  and the results on Grassmannians needed  throughout
this paper.  Section 3 contains the main results of this
paper: a full classification of $\GL(V)$-invariant initialized Ulrich bundles on $\Gr(k,n)$ (Theorem
\ref{main}). As a consequence, we establish the minimum rank of any $\GL(V)$-invariant Ulrich bundle on $\Gr(k,n)$ and we determine the slope  of any Ulrich bundle on $\Gr(k,n)$.

\vspace{3mm}

\noindent {\em Acknowledgment:} We deeply thank Frank-Olaf Schreyer for many insights that lead us to the main result of the paper.

%%%%%%%%%%%%%%%%%%%%%%%%%%%%%%%%%%%%%%%%%%%%%%%%%%%%%%%%%%%%%%%%%%%%%%%%%%%%%

\section{Preliminaries}

In this section we are going to introduce the definitions  and main properties of Grassmannians as well as those of Ulrich sheaves that will be used throughout the rest of the paper.

Let us start fixing some notation. We will work over an algebraically closed field $K$ of characteristic zero. Given a
non-singular variety $X$ equipped with an ample line bundle $\odi{X}(1)$, the line bundle $\odi{X}(1)^{\otimes l}$ will
be denoted by $\odi{X}(l)$. For any coherent sheaf $\shE$ on $X$ we are going to denote the twisted sheaf
$\shE\otimes\odi{X}(l)$ by $\shE(l)$ and the slope of $\shE$ by $\mu(\shE):= \deg(c_1(\shE))/\rk(\shE)$. As usual,
$\Hl^i(X,\shE)$ stands for the cohomology groups, $\h^i(X,\shE)$ for their dimension and $\Hl^i_*(X,\shE)=\oplus _{l
\in \ZZ}\Hl^i(X,\shE(l))$.

\vspace{4mm}

\begin{definition}\rm
Let $(X,\odi{X}(1))$ be a polarized variety. A coherent sheaf $\shE$ on $X$ is \emph{Arithmetically Cohen Macaulay}
(ACM for short) if it is locally Cohen-Macaulay (i.e., $\depth \shE_x=dim \odi{X,x}$ for every point $x\in X$) and has
no intermediate cohomology:
$$
\Hl^i_*(X,\shE)=0 \quad\quad \text{    for all $i=1, \ldots , dim X-1.$}
$$
\end{definition}
Notice that when $X$ is a non-singular variety, which is going to be our case, any coherent ACM sheaf on $X$ is locally
free. For this reason we are going to speak uniquely of ACM bundles.

\begin{definition}\rm
Given a polarized variety $(X,\odi{X}(1))$, a coherent sheaf $\shE$ on  $X$
  is \emph{initialized} if
$$
\Hl^0(X,\shE(-1))=0 \ \ \text{ but } \  \Hl^0(X,\shE)\neq 0.
$$
Notice that when $\shE$ is a locally Cohen-Macaulay sheaf, there always exists an integer $t_0$ such that
$\shE_{init}:=\shE(t_0)$ is
initialized.
\end{definition}

The ACM bundles that we are interested in share a stronger property, namely they have the maximal possible number of
global sections, i.e. they realize the upper bound  given by the following result:

\vspace{3mm}

\begin{theorem}
Let $X\subseteq\PP^n$ be an integral subscheme and let $\shE$  be an initialized ACM sheaf on $X$. Denote by $m(\shE)$
the minimal number of generators  of the $R_X$-module $\Hl^0_*(\shE)$. Then,
$$\h^0(\shE)\leq m(\shE)\leq \deg (X)\rk(\shE).
$$
\end{theorem}
\begin{proof} See \cite{CH}; Theorem 3.1.
\end{proof}

\vspace{3mm}

\begin{definition} \rm Given a projective scheme $X\subseteq \PP^n$ and a coherent sheaf $\shE$ on $X$, we say that
$\shE$ is an \emph{Ulrich sheaf} if  $\shE$ is an ACM sheaf and $\h^0(\shE_{init})=\deg(X)\rk(\shE)$, i.e. the minimal
number of generators of Ulrich sheaves is as large as possible.
\end{definition}

\vspace{3mm}

 Modules attaining this upper bound were studied by Ulrich in \cite{Ulr}. It is of particular interest for us the following fact:

\begin{proposition} \label{ulrichslope}
Let $X\subseteq\PP^n$ be a nonsingular $d$-dimensional ACM variety and let $\cE$ be an initialized sheaf on $X$. Then, the following conditions are equivalent:
\begin{enumerate}
\item[(i)] $\shE$ is Ulrich.
\item[(ii)] $\Hl^i(\shE(-i))=0$ for $i>0$ and $\Hl^i(\shE(-i-1))=0$ for $i<d$.
\item[(iii)] For some (resp. all) finite linear projection $\pi:X \rightarrow\PP^d$, the sheaf $\pi_*\shE$ is the trivial
    sheaf $\odi{\PP^d}^t$ for some $t$.
\end{enumerate}

In particular, all initialized Ulrich bundles on $X$ have the same slope.

\end{proposition}
\begin{proof}

For the equivalence between (i), (ii) and (iii) see \cite{ESW}, Proposition 2.1.

 The last assertion follows from \cite{ESW}, Corollary 2.2 and the fact that the slope of any vector bundle $\cE$ on
 $X$ is encoded in the coefficient of $t^{d}$ in $\chi(\cE(t))$.
\end{proof}

\begin{remark}
\label{pu}
Let $X\subseteq\PP^n$ be a nonsingular $d$-dimensional ACM variety and let $\cE$ an initialized sheaf on $X$. It
follows from the above Proposition that $\cE$ is an Ulrich sheaf if and only if
\begin{equation} \label{condicioUlrich}
\Hl^i(\cE(t))=0 \quad \mbox{for any $i\geq 0$ and $-d \leq t \leq -1$.  }
\end{equation}
\end{remark}

\vspace{3mm}

The search of Ulrich sheaves on a particular variety is a challenging problem. In fact, few examples of varieties
supporting Ulrich sheaves are known, although in \cite{ESW} has been conjectured that any variety $X$ supports an
Ulrich sheaf. Moreover in \cite{ESW} they also ask what is the smallest possible rank of an Ulrich bundle on $X$.  In
this paper, we are going to focus our attention on the existence of $\GL(V)$-invariant Ulrich  bundles on Grassmannians
$\Gr(k,n)$. In \cite{ESW}, Corollary 5.7 the authors prove the existence of a unique $\GL(V)$-invariant Ulrich bundle
on $\PP^n=\PP(V)$ for the $d$-tuple embedding; it has  rank $d^{{n \choose 2}}$.  In this paper,  we give  the full
list of   $\GL(V)$-invariant Ulrich  bundles
on $\Gr(k,n)$ and, as a by product,  the smallest possible rank of a $\GL(V)$-invariant Ulrich
 bundle on $\Gr(k,n)$ is explicitly given in Corollary \ref{minimrankUlrich}.

\vskip 3mm

Let us recall now the basic facts on $\Gr(k,n)$ needed in this paper. Let  $\PP^n=\PP(V)$ be the $n$-dimensional projective space associated to the   $(n+1)$-dimensional vector space $V$ and as usual we denote by $\Gr(k,n)$ ($G$ for short) the Grassmann variety of
  $k$-linear subspaces of $\PP^n$ together with the Pl\"{u}cker embedding
\[ \Gr(k,n) \longrightarrow \PP(\bigwedge^{k+1}V)\cong \PP^{{n+1\choose k+1}-1}.\]

\noindent Since $\Gr(k,n) \cong \Gr(n-k-1,n)$ we can assume that $k \leq \frac{n-1}{2}$.
Recall that $\Gr(k,n)$ is an Arithmetically Cohen-Macaulay variety, $\dim(\Gr(k,n))=(k+1)(n-k)$ and
\begin{equation}
 \de(\Gr(k,n))= \frac{((k+1)(n-k))! k!(k-1)! \cdots 2!}{n!(n-1)! \cdots (n-k)!}
\end{equation}
(see for instance \cite{Mukai}, Proposition 1.10). On $\Gr(k,n)$ there is the universal exact sequence
\[ 0 \rightarrow  \cS^{\vee}\rightarrow V \otimes \cO_G \rightarrow \cQ \rightarrow 0 \]
defining the universal vector bundles $\cS$ and $\cQ$ of ranks $n-k$ and $k+1$, respectively. Note that
$\bigwedge^{k+1}\cQ \cong \bigwedge^{n-k}\cS$ is the positive generator $\cO_G(1)$ of $\Pic(\Gr(k,n))$. The tangent
bundle is
$\cQ \otimes \cS$ and hence the canonical line bundle is $\cO_G(-n-1)$.

\vspace{3mm}
Given any rank $r$ vector bundle $\cE$ on $\Gr(k,n)$ (resp. any $r$-dimensional vector space $W$) we will denote by
$\Sigma^{\lambda}\cE$ (resp. $\Sigma^{\lambda}W$ ) the result of applying the Schur functor $\Sigma^{\lambda}$ with
 $\lambda=(\lambda_1, \cdots, \lambda_r) \in \ZZ^{r}$  a partition with $r$ parts, i.e., $\lambda_1 \geq \lambda_2 \geq
 \cdots \geq \lambda_r \geq 0 $.  We choose the conventions such that
\[ \Sigma^{(\lambda_1,0,\cdots,0)}\cE = S^{\lambda_1}\cE  \quad \quad (\mbox{resp. } \Sigma^{(\lambda_1,0,\cdots,0)}W =
S^{\lambda_1}W)\]
is the $\lambda_1$-symmetric power of $\cE$ (resp. $W$) and
\[ \Sigma^{(\tiny{\overbrace{1,\cdots,1}^{m}},0,\cdots,0)}\cE = \bigwedge^m\cE  \quad \quad (\mbox{resp. }
\Sigma^{(\tiny{\overbrace{1,\cdots,1}^{m}},0,\cdots,0)}W = \bigwedge^mW)\]
is the $m$th-exterior power of $\cE$ (resp. $W$).  We also use the convention $\Sigma^{(\lambda_1, \cdots,
\lambda_s)}\cE=\Sigma^{(\lambda_1, \cdots, \lambda_s, 0, \cdots,0)}\cE$ whenever $s<r$. For any $t \in \ZZ$, we have
$\Sigma^{\lambda}\cQ \otimes \cO_G(t)=\Sigma^{\lambda+t}\cQ$ where $\lambda+t:=\lambda+(t, \cdots,
t)=(\lambda_1+t,\cdots,\lambda_{k+1}+t)$. We will draw the Young diagram corresponding to any partition $\lambda$ by
putting $\lambda_i$ boxes in the $i$-th row and left justifying the picture. We will denote also by $\lambda$ the
corresponding Young diagram and by $\lambda^{t}$ its transpose. For example the partition $(8,5,3,2,2,0,0)$ corresponds
to
\vspace{5mm}
\smallbreak
\centerline{
\vbox{
\def\star{\rlap{\hbox to 13pt{\hfil\raise3.5pt\hbox{$*$}\hfil}}}
\def\ {\hbox to 13pt{\vbox to 13pt{}\hfil}}
\def\*{\star\ }
\def\_{\hbox to 13pt{\hskip-.2pt\vrule\hss\vbox to 13pt{\vskip-.2pt
            \hrule width 13.4pt\vfill\hrule\vskip-.2pt}\hss\vrule\hskip-.2pt}}
\def\x{\star\_}
\offinterlineskip
\hbox{\ \ \    \_\_\_\_\_\_\_\_\ \ \ }
\hbox{\ \ \    \_\_\_\_\_\ \ \ \ \ \ }
\hbox{\ \ \    \_\_\_\ \ \ \ \ \ \ \ }
\hbox{\ \ \     \_\_\ \ \ \ \ \ \ \ \ }
\hbox{\ \ \     \_\_\ \ \ \ \ \ \ \ \ }
\hbox{\ \ \   \  \  \  \  \  \  \  \  \ \ \ }
}
}

\vspace{5mm}
\begin{remark}
For any integer $l \in \ZZ$, we have $$\Sigma^{(b_1,\cdots,b_{k+1})}Q \otimes \Sigma^{(a_1,\cdots,a_{n-k})}S^{\vee}
\cong \Sigma^{(b_1+l,\cdots,b_{k+1}+l)}Q \otimes \Sigma^{(a_1+l,\cdots,a_{n-k}+l)}S^{\vee}.$$ Thus, from now on while
dealing with $\GL(V)$-invariant bundles $\Sigma^{(b_1,\cdots,b_{k+1})}Q \otimes \Sigma^{(a_1,\cdots,a_{n-k})}S^{\vee}$
on $\Gr(k,n)$ we will  assume that $a_{n-k}=0$.
\end{remark}

\vspace{5mm}

For a later use, we collect here the following  well know result

\begin{lemma}
\label{dim}
Let $\cE$ be a rank $r$ vector bundle on $\Gr(k,n)$ (resp. let $W$ be any $r$-dimensional vector space) and
$\lambda=(\lambda_1, \cdots, \lambda_r) \in \ZZ^{r}$ a partition with $r$ parts. Then
\[\rk(\Sigma^{\lambda}\cE)=\dim_K(\Sigma^{\lambda}W)= \prod_{(i,j) \in \cF_{\lambda}}\frac{(r+j-i)}{d_{ij}} \]
where $\cF_{\lambda}= \{(i,j) | (i,j)\mbox{ is a position in the Young diagram } \lambda  \}$ and
$$d_{ij}=\lambda_i+(\lambda^{t})_j-(i+j)+1.$$
\end{lemma}

\vspace{3mm}

 Recall that every irreducible $\GL(V)$-invariant vector bundle on $\Gr(k,n)$ is isomorphic to $\Sigma^{\beta} \cQ
 \otimes \Sigma^{\gamma}\cS^{\vee}$ for some non-increasing  $\beta =(b_1, \cdots , b_{k+1})\in \ZZ^{k+1}$ and $\gamma
 =(a_1,\cdots,a_{n-k})\in \ZZ^{n-k}$. The Borel-Bott-Weil Theorem is a powerful tool that, in particular, computes the
 cohomology of $\GL(V)$-equivariant vector bundles on $\Gr(k,n)$. To state it, recall that   for every $\alpha \in
 \ZZ^{n+1}$ there exists an element $\sigma $ of the Weyl group $S_{n+1}$ of $\GL(V)$ such that $\sigma(\alpha)$ is non
 increasing and it is unique if and only if all the entries of $\alpha$ are distinct. Denote by
 $\rho=(n+1,n,n-1,\cdots,2,1)$ the half sum of the positive roots of $\GL(V)$ and by $l:S_n \rightarrow \ZZ$ the
 standard length function.

\begin{theorem}{\bf (Borel-Bott-Weil)}
\label{bott}
Let $\beta =(b_1, \cdots , b_{k+1})\in \ZZ^{k+1}$ and $\gamma =(a_1,\cdots,a_{n-k})\in \ZZ^{n-k}$ be two non-increasing
sequences  and let $\alpha=(\beta,\gamma) \in \ZZ^{n+1}$ be their concatenation.

\begin{itemize}
\item[(a)] Assume that all entries of $\alpha + \rho $ are distinct and let $\sigma$ be the unique permutation such
    that $\sigma(\alpha+\rho)$ is strictly decreasing. Then

\[ \Hl^m(\Gr(k,n), \Sigma^{\beta} \cQ \otimes \Sigma^{\gamma}\cS^{\vee})= \left \{ \begin{array}{ll}
\Sigma^{\sigma(\alpha+\rho)-\rho} V^* & \mbox{if} \quad m=l(\sigma) \\
0 & \mbox{otherwise.}
 \end{array} \right. \]
\item[(b)] If at least two entries of $\alpha + \rho $ coincide then $$
\Hl^{\bullet}(\Gr(k,n), \Sigma^{\beta} \cQ \otimes \Sigma^{\gamma}\cS^{\vee})= 0.$$
\end{itemize}
\end{theorem}

\vspace{3mm} Notice that Theorem \ref{bott}, in particular, implies that any twist of an irreducible $\GL(V)$-invariant
vector bundle has at most cohomology concentrated in one degree.

%%%%%%%%%%%%%%%%%%%%%%%%%%%%%%%%%%%%%%%%%%%%%%%%%%%

%%%%%%%%%%%%%%%%%%%%%%%%%%%%%%%%%%%%%%%%%%%%%%%%%%%%%%%%%%%%%%%%%%%%%%%%%%%%%

\section{Characterization of Ulrich GL(V)-invariant bundles on $\Gr(k,n)$}

This section contains the main new result of this work, namely, the complete classification of   all irreducible $\GL(V)$-invariant Ulrich bundles on $\Gr(k,n)$. In particular, we have  that  any Grassmann variety $\Gr(k,n) \hookrightarrow \PP(\bigwedge^{k+1} V)$ embedded by the Pl\"{u}cker embedding supports an Ulrich bundle giving  for the case of Grassmann varieties, an affirmative answer to Problem \ref{problem} (a).

\vspace{2mm}

 Let us start with two technical results.

\vspace{3mm}

\begin{remark}
\label{ordenacio}
Let $\cE=\Sigma^{(b_1,\cdots,b_{k+1})}Q \otimes \Sigma^{(a_1,\cdots,a_{s})}S^{\vee}$ be an irreducible
$\GL(V)$-invariant vector bundle on the Grassmann variety $\Gr(k,n)$. Notice that if $\cE$ is initialized, we have
$\Hl^0(\Gr(k,n),\cE)\ne 0$ and
$\Hl^{i}(\Gr(k,n),\cE)=0$ for any $i>0$. Therefore, applying Theorem \ref{bott}, if $\cE$ is initialized we have
\[ b_1 \geq \cdots \geq b_{k+1}=a_1 \geq \cdots \geq a_s>0. \]
\end{remark}

\vspace{3mm}

To compute the slope of any Ulrich bundle on a $\Gr(k,n)$ we will need the following result

\vspace{3mm}

\begin{lemma}
\label{slope}
Let $\beta =(\beta_1, \cdots ,\beta_{k+1})\in \ZZ^{k+1}$ and $\gamma =(\gamma_1,
\cdots+\gamma_{n-k})\in \ZZ^{n-k}$ be two non-increasing sequences. Then,
\[ \mu(\Sigma^{\beta} \cQ \otimes \Sigma^{\gamma}\cS^{\vee}) = \Big ( \frac{|
\beta_1+ \cdots+\beta_{k+1}|}{k+1}-\frac{|\gamma_1+\cdots+\gamma_{n-k}|}{n-k} \Big ) \cdot d \]
with $d:=\de(\Gr(k,n))= \frac{((k+1)(n-k))! k!(k-1)! \cdots 2!}{n!(n-1)! \cdots (n-k)!}$.
\end{lemma}
\begin{proof} Recall that for any rank $r$ vector bundle $\cE$ on $\Gr(k,n)$ and any
 non-increasing sequence $\lambda \in \ZZ^r$, $\mu(\Sigma^{\lambda}\cE)= |\lambda_1+\cdots+\lambda_r|\cdot \mu(\cE)$
 (\cite{Rubei}, Pg. 4). Hence, since by the multiplicative character of the Chern classes
\[ \begin{array}{lll}\mu(\Sigma^{\beta} \cQ \otimes \Sigma^{\gamma}\cS^{\vee}) & = &
\frac{\de(c_1(\Sigma^{\beta} \cQ \otimes \Sigma^{\gamma}\cS^{\vee})) }{\rk(\Sigma^{\beta} \cQ \otimes
\Sigma^{\gamma}\cS^{\vee})}   \\
& = &
 \frac{\de(c_1(\Sigma^{\beta} \cQ )\rk(\Sigma^{\gamma}\cS^{\vee})+c_1(
  \Sigma^{\gamma}\cS^{\vee})\rk (\Sigma^{\beta} \cQ ) )}{\rk(\Sigma^{\beta}
 \cQ){\cdot}\rk(\Sigma^{\gamma}\cS^{\vee})}   \\
 & = &
\mu(\Sigma^{\beta} \cQ)+ \mu( \Sigma^{\gamma}\cS^{\vee}),\end{array}\] the result follows from the fact that $\mu(\cQ)=
\frac{d}{k+1}$ and $\mu(\cS^{\vee})=- \frac{d}{n-k}$.
\end{proof}
\vspace{3mm}

Before facing the main result, we are going to give a necessary condition for a $\GL(V)$-invariant
vector bundle to be an Ulrich bundle.

\begin{proposition}
\label{b1}
Let $\cE=\Sigma^{(b_1,\cdots,b_{k+1})}Q \otimes \Sigma^{(a_1,\cdots,a_{s})}S^{\vee}$ be an irreducible initialized
$\GL(V)$-invariant vector bundle on the Grassmann variety $\Gr(k,n)$. If $\cE$ is an initialized Ulrich bundle then
\[b_1=k(n-k-1). \]
\end{proposition}
\begin{proof}
Recall that for any integer $t \in \ZZ$,
\[\cE(t)=\Sigma^{(b_1+t,\cdots,b_{k+1}+t)}Q \otimes \Sigma^{(a_1,\cdots,a_{s})}S^{\vee} \]
and by Proposition \ref{ulrichslope}, $(b)$ we have
\[ \Hl^d(X,\cE(-(k+1)(n-k)))=0 \quad \mbox{and } \Hl^d(X,\cE(-(k+1)(n-k)-1)) \neq 0\]
where $d=\dim(\Gr(k,n))=(k+1)(n-k)$. It follows from Theorem \ref{bott} that the first vanishing forces $b_1 \geq
k(n-k-1)$ while the fact that
$\Hl^d(X,\cE(-(k+1)(n-k)-1)) \neq 0$ implies $b_1 \leq k(n-k-1)$. Putting together both inequalities we get
$b_1=k(n-k-1)$.
 \end{proof}

\vspace{3mm}

The following example will illustrate the idea behind the main result.

\vspace{5mm}

\begin{example}
\label{mainexample}
We consider the Grassmannian $\Gr(5,17)$, i.e. $k+1=6$ and $n-k=12$, together with two ordered sequences of integers of the same length
\[(k_1,k_2)=(2,3) \quad \mbox{and} \quad (n_1,n_2)=(3,4)\]
such that $k+1=k_1\cdot k_2$ and $n-k=n_1\cdot n_2$. We are going to  associate  univocally to them a
$\GL(V)$-invariant Ulrich bundle $\cE=\Sigma^{(b_1,\cdots,b_6)}\cQ \otimes \Sigma^{(a_1, \cdots,a_{12})} \cS^{\vee}$ on
$\Gr(5,17)$.  Given the pair $(n_1,k_1)=(3,2)$ we consider the $3 \times 2$ block $B_1$
with 3 columns and 2 rows filled in by integers in the following way
\vspace{3mm}
$$
\centerline{
\begin{tabular}{|c|c|c|} \hline
4 & 5 & 6  \\  \hline 1 & 2 & 3 \\ \hline
\end{tabular}}
$$
\[ \mbox{{\footnotesize Block of type $B_1$}} \]

Now, we consider the next pair $(n_2,k_2)=(4,3)$ and we built the block $B_2$ with 3 rows of 4 blocks of type $B_1$ on
each row. Place first the 4 blocks of the first row, then the 4 blocks of the second row and finally the 4 blocks of
the last row, each block filled in by integers as shown in the picture:

\vspace{3mm}

$$
\centerline{
\begin{tabular}{|c|c|c|| c | c| c || c|c|c|| c | c| c |}
\hline
52 & 53 & 54 & 58 & 59 & 60 & 64 & 65 & 66 & 70 & 71 & 72   \\
\hline
49 & 50 & 51 & 55 & 56 & 56 & 61 & 62 & 63 & 67 & 68 & 69 \\
\hline \hline 28 & 29 & 30 & 34 & 35 & 36 & 40 & 41 & 42 & 46 & 47 & 48 \\
\hline 25 & 26 & 27 & 31 & 32 & 33 & 37 & 38 & 39 & 43 & 44 & 45\\
\hline \hline 4 & 5 & 6 & 10 & 11 & 12 & 16 & 17 & 18 & 22& 23 &24 \\
\hline 1& 2& 3 & 7 & 8 & 9 & 13& 14 & 15 & 19 & 20 & 21 \\
\hline
\end{tabular}}
$$
\[ \mbox{{\footnotesize Block of type $B_2$}} \]

We define two non-increasing sequences of integers
{ \scriptsize
\[\begin{array}{ll}(a_1,a_2,a_3,a_5,a_5,a_6,a_7,a_8,a_9,a_{10},a_{11},a_{12}) & := (21,20,19,15,14,13,9,8,7,3,2,1)-
(12,11,10,9,8,7,6,5,4,3,2,1)\\ & = (9,9,9,6,6,6,3,3,3,0,0,0)\end{array}\]
\[\begin{array}{ll}(b_1,b_2,b_3,b_4,b_5,b_6) & := (52,49,28,25,4,1)-(18,17,16,15,14,13)+21(1,1,1,1,1,1) \\ & =
(55,53,33,31,11,9). \end{array}\]}

Let us see that $\cE=\Sigma^{(b_1,\cdots,b_6)}\cQ \otimes \Sigma^{(a_1, \cdots,a_{12})} \cS^{\vee}$ is an Ulrich bundle on the
72-dimensional Grassmannian $\Gr(5,17)$. First of all notice that $\cE$ is an initialized vector bundle. Indeed, since $b_6=a_1$ and $(b_1,b_2, \cdots,b_6,a_1, \cdots, a_{12})$ is a non-increasing sequence by Theorem \ref{bott}, $\Hl^0(\cE(-1))=0$ and $\Hl^0(\cE) \neq 0$. Hence, according to Remark
\ref{pu}, we only need to check that for any integer $t \in  [1,72 ]$,
\[ \Hl^i(\cE(-t))=0 \quad 0 \leq i \leq 72.\]
To this end, we observe that any integer $t \in [1,72]$ corresponds to a unique point of position $(i,7-j)$ for some $1
\leq i \leq 12$, $1 \leq j \leq 6$ inside the block $B_2$. Moreover, if $t$ is in the position $(i,7-j)$ then
\[b_j+17-j+2-t=a_i+12-(i-1). \]
That means that for any integer $t \in [1,72]$ at least two entries of
{\scriptsize
\[(b_1-t,b_2-t,b_3-t,b_4-t,b_5-t,b_6-t,a_1,a_2,a_3,a_5,a_5,a_6,a_7,a_8,a_9,a_{10},a_{11},a_{12}) +\] \[
(18,17,16,15,14,13,12,11,10,9,8,7,6,5,4,3,2,1)\]}
coincide. Hence, by Theorem \ref{bott} this means that for
for any integer $t \in  [1,72 ]$,
\[ \Hl^i(\cE(-t))=0, \quad 0 \leq i \leq 72.\]
Therefore, $\cE$ is an Ulrich bundle on $\Gr(5,17)$.
\end{example}

\vspace{3mm}
Let us fix a definition that we will use in the statement of the main result.

\begin{definition}
\label{notacioteorema}
For any $k,n \in \NN $ and for any pair of ordered sequences of positive integers of the same length $(k_1, \cdots,k_s)$ and
$(n_1, \cdots,n_s)$ with $n_l >1$ for $l >1$, $k_l > 1$ for $l < s$,
\[ \begin{array}{l} k+1=k_1 \cdot k_2 \cdots k_s, \mbox{ and}\\
 n-k=n_1 \cdot n_2 \cdots n_s, \end{array} \]
we define:
\[ \alpha= n_1+\sum_{i=1}^{s-1} (n_{s-i+1}-1)k_1 \cdots k_{s-i} \cdot n_1 \cdots n_{s-i}.\]

For any integer $j$, $1 \leq j \leq k+1$, if
$$j=j_1\cdot k_1\cdots k_{s-1}+ j_2\cdot k_1\cdots k_{s-2}+ \cdots +j_{s-2}k_1k_2+j_{s-1}k_1+\rho_j$$
with $1 \leq \rho_j \leq k_1$ and $0 \leq j_l < k_{s-l+1}$ for $1 \leq l \leq s-1$ we define
\[b_{k+2-j}=\alpha+(\rho_j-1)n_1+1+j_{s-1}k_1n_1n_2+j_{s-2}k_1k_2n_1n_2n_3+ \cdots+j_1k_1k_2\cdots k_{s-1}n_1 \cdots
n_s. \]

For any integer $i$, $1 \leq i \leq n-k$, if
$$i=i_1\cdot n_1\cdots n_{s-1}+ n_2\cdot n_1\cdots n_{s-2}+ \cdots +i_{s-2}n_1n_2+i_{s-1}n_1+\rho_i$$
with $1 \leq \rho_i \leq n_1$ and $0 \leq i_l \leq n_{s-l+1}$ for $1 \leq l \leq s-1$ we define
\[a_{n-k+1-j}=\rho_i+1-i+i_{s-1}k_1n_1+i_{s-2}k_1k_2n_1n_2+ \cdots+i_1k_1k_2\cdots k_{s-1}n_1 \cdots n_{s-1}. \]
Finally, we define by
\[ \cE^{(k_1, \cdots,k_s)}_{(n_1, \cdots,n_s)} := \Sigma^{(b_1, \cdots,b_{k+1})} \cQ \otimes \Sigma^{(a_1,
\cdots,a_{n-k})} \cS^{\vee} \]
the corresponding $\GL(V)$-invariant vector bundle on $\Gr(k,n)$ with the $a_i$, $1 \leq i \leq n-k$ and $b_j$, $1 \leq
j \leq k+1$ given as above.
\end{definition}

\vspace{3mm}
\begin{theorem}
\label{main} Let $\Gr(k,n) \hookrightarrow \PP(\bigwedge^{k+1} V)$ be the Grassmannian embedded by the Pl\"{u}cker
embedding. Then, for any pair of ordered sequences of positive integers of the same length $(k_1, \cdots,k_s)$ and $(n_1,
\cdots,n_s)$  such that $n_l >1$ for $l >1$, $k_l > 1$ for $l < s$,
\[ k+1=k_1 \cdot k_2 \cdots k_s, \text{ and }  n-k=n_1 \cdot n_2 \cdots n_s,\]
the $\GL(V)$-invariant vector bundle
\[ \cE=\cE^{(k_1, \cdots,k_s)}_{(n_1, \cdots,n_s)}  \]
is an Ulrich bundle on $\Gr(k,n)$. Conversely, any initialized $\GL(V)$-invariant Ulrich bundle on $\Gr(k,n)$ is of
this form for some ordered sequences of positive integers of the same length $(k_1, \cdots,k_s)$ and $(n_1, \cdots,n_s)$  with $n_l >1$ for $l >1$, $k_l > 1$ for $l < s$,
\[ k+1=k_1 \cdot k_2 \cdots k_s, \text{ and }  n-k=n_1 \cdot n_2 \cdots n_s.\]
\end{theorem}
\begin{proof}
First of all notice that $\cE$ is an initialized vector bundle. Indeed, since $b_{k+1}=a_1$, by Theorem \ref{bott} we have $\Hl^0(\cE(-1))=0$ and since \[ (b_1+n+1, \cdots,b_{k+1}+n-k+1,a_1+n-k, \dots, a_{n-k}+1)\] is a strictly decreasing  sequence we have
$\Hl^0(\cE) \neq 0$. Hence, according to Remark \ref{pu}, $\cE$ is an Ulrich
bundle if and only if for any integer $t \in [1,(k+1)(n-k)]$,
 \[ \Hl^i(\cE(-t))=0,  \quad 0 \leq i \leq (k+1)(n-k).\]

In order to prove all these vanishing, we will proceed as in Example \ref{mainexample}.  Given the pair $(n_1,k_1)$
we consider the $n_1 \times k_1$ block $B_1$
with $n_1$ columns and $k_1$ rows filled in by integers in the following way
\vspace{3mm}
$$
\centerline{
\begin{tabular}{|c| c | c| c | c| c |}
\hline
${\footnotesize(k_1-1)n_1+1} $ &  ${\footnotesize(k_1-1)n_1+2} $ & \ldots & \ldots & \ldots & $k_1n_1$   \\
\hline
\vdots & \vdots &  &  &  & \vdots \\
\hline  \vdots  & \vdots &  &  &   & \vdots \\
\hline $2n_1+1$  & $2n_1+2$ & \ldots & \ldots & \ldots  & $3n_1$\\
\hline $n_1+1$   & $n_1+2$ & \ldots & \ldots & \ldots  & $2n_1$ \\
\hline 1& 2 & \ldots & \ldots & \ldots  & $n_1$ \\
\hline
\end{tabular}}
$$
\[ \mbox{\footnotesize {Block of type $B_1$}} \]

For any integer $i$, $2 \leq i \leq s$, we consider the pair $(n_i,k_i)$ and we built the block $B_i$ with $k_i$ rows of $n_i$ blocks of
type $B_{i-1}$ on each row. Place first the $n_i$ blocks of the first row, then the $n_i$ blocks of the second row and
so on until the $k_i$-th row (see Example \ref{mainexample}).
\vspace{3mm}

 Now we observe that any integer  $t \in [1,(k+1)(n-k)]$ corresponds to a unique point of position $(i,k+2-j)$ for some
 $1 \leq i \leq n-k$, $1 \leq j \leq k+1$ inside the block $B_s$. Moreover, if $t$ is in the position $(i,k+2-j)$ then
\[b_j+n-j+2-t=a_i+(n-k)-(i-1). \]
That means that for any integer $t \in [1,(k+1)(n-k)]$ at least two entries of
\[(b_1-t,b_2-t,\cdots, b_{k+1}-t,a_1,a_2,\cdots, a_{n-k}) + (n+1,n,\cdots,2,1)\]
coincide. Hence, by Theorem \ref{bott},
for any integer $t \in  [1,(k+1)(n-k)]$,
\[ \Hl^i(\cE(-t))=0, \quad 0 \leq i \leq (k+1)(n-k).\]
Therefore, $\cE$ is an Ulrich bundle on $\Gr(k,n)$.

\vspace{3mm}

Let us prove the converse. Let $\cE=\Sigma^{(b_1, \cdots,b_{k+1})} \cQ \otimes \Sigma^{(a_1, \cdots,a_{n-k})}
\cS^{\vee}$ be an initialized $\GL(V)$-invariant Ulrich bundle on $\Gr(k,n)$. By Remark \ref{pu}, for any integer $t
\in [1,(k+1)(n-k)]$,
 \[ \Hl^i(\cE(-t))=0,  \quad 0 \leq i \leq (k+1)(n-k).\]
 Thus, by Theorem \ref{bott}, for any integer $t \in [1,(k+1)(n-k)]$, there exists a pair $(i,j)$ such that
 \begin{equation} \label{cond} b_j+n-j+2-t=a_i+(n-k)-(i-1). \end{equation}
 Consider the $(n-k)\times(k+1)$ block $B$ with $(n-k)$ columns and $(k+1)$ rows, put on the position $(i,k+2-j)$
 the value of $t$ such that (\ref{cond}) holds and denote by $t_{i,j}$ that value, that is:
 \begin{equation} \label{cond2} b_j+n-j+2-t_{i,j}=a_i+(n-k)-(i-1). \end{equation}
 Notice that for any integer $i$, $1 \leq i \leq n-k$, and any integer $j$, $1 \leq j \leq k+1$ we have:
 \begin{equation} \label{cond3}  t_{i,j+1}-t_{i,j}=b_{j+1}-b_j-1 \quad \mbox{and} \quad t_{i+1,j}-t_{i,j}=a_i-a_{i+1}+1 \end{equation}
and that there is a bijection between the integers  $t \in [1,(k+1)(n-k)]$ and the values inside
 the block $B$.

 It follows from (\ref{cond2}) that $t_{1,k+1}=1$. Denote by $n_1$, $1 \leq n_1 \leq n-k$, the integer such that $t_{i,k+1}=i$ for $1 \leq i \leq n_1$ and $t_{n_1+1,k+1} > n_1+1$. If $n_1=n-k$, take $s=1$, $k_1=k+1$ and the relations (\ref{cond2}) give us $\cE=\cE^{(k+1)}_{(n-k)}$. Assume $n_1 < n-k$. Then by (\ref{cond2}) and (\ref{cond3}) we must have $t_{1,k}=n_1+1$. By the same reason, the relations (\ref{cond2}) and (\ref{cond3}) force to have $t_{i,k}=n_1+i$ for $1 \leq i \leq n_1$. Denote by $k_1$ the integer such that
$t_{1,k+2-k_1}=(k_1-1)n_1+1$ and $t_{1,k+1-k_1} > k_1n_1+1$. Now, we must have $t_{n_1+1,k+1}=k_1n_1+1$ and since we still have the same relations $(\ref{cond2})$, we must fill in the same way the next $n_1 \times k_1$ box, that is the $t_{i,j}$ with $n_1+1 \leq i \leq 2n_1$ and $k_1 \leq j \leq k+2$. We will repeat this $n_2$ times, where $n_2$ is the integer such that $t_{1,k+1-k_1}=k_1n_1n_2+1$. Pushing forward the same argument and strongly using the fact that  we have a bijection between the integers $t \in [1, (k+1)(n-k) ]$ and the values of the $t_{i,j}$ inside $B$ we get the existence of positive integers $n_1, \cdots, n_s$ and $k_1, \cdots, k_s$ such that
 \[ k+1=k_1 \cdot k_2 \cdots k_s,\]
\[ n-k=n_1 \cdot n_2 \cdots n_s\]
$n_l >1$ for $l >1$, and $k_l > 1$ for $l < s$.  Moreover, according to the values that the $t_{i,j}$ have taken, the relations (\ref{cond2}) also give us
$ \cE=\cE^{(k_1, \cdots,k_s)}_{(n_1, \cdots,n_s)}$.
\end{proof}
\vspace{3mm}

\begin{remark} It should be mentioned that all the $\GL(V)$-invariant vector bundles $\cE=\Sigma^{(b_1,\cdots,
b_{k+1})}\cQ \otimes \Sigma^{(a_1, \cdots, a_s)} \cS^{\vee}$ are simple and hence indecomposable on $\Gr(k,n)$. So, all Ulrich bundles so far constructed are indecomposable. \end{remark}

\vspace{3mm}

\begin{example} (1) Let us consider the Grassmann variety $\Gr(1,21)$.  By Theorem \ref{main} (see also Corollary \ref{noUlrich} (a)) it supports 6 initialized $\GL(V)$-invariant Ulrich bundles. Indeed, using the notation
$(a^n,b^m):=(\overbrace{a,...a}^{n},\overbrace{b,..,b}^{m})$, they are:
\[ \begin{array}{l} \cE^{(2)}_{(20)} =\Sigma^{(19)}\cQ \\
\cE^{(2, 1)}_{(10, 2)}=\Sigma^{(19,10)}\cQ \otimes \Sigma^{(10^{10},0^{10})}\cS^{\vee}, \\
\cE^{(2, 1)}_{(5 , 4)}=\Sigma^{(19,15)}\cQ \otimes \Sigma^{(15^5,10^5,5^5,0^5)} \cS^{\vee}, \\
\cE^{(2 , 1)}_{(4 , 5)}=\Sigma^{(19,16)}\cQ \otimes \Sigma^{(16^4,12^4,8^4,4^4,0^4)}\cS^{\vee} , \\
\cE^{(2 , 1)}_{(2, 10)}=\Sigma^{(19,18)}\cQ \otimes
\Sigma^{(18^2,16^2,14^2,12^2,10^2,8^2,6^2,4^2,2^2,0^2)}\cS^{\vee} \quad \mbox{and} \\
\cE^{(2 , 1 )}_{(1 , 20)}=\Sigma^{(19,19)}\cQ \otimes \Sigma^{(19, 18,
17,16,15,14,13,12,11,10,9,8,7,6,5,4,3,2,1,0)}\cS^{\vee} . \end{array} \]

(2) On $\Gr(11,29)$, the sequences $(2,2,3)$ and $(2,3,3)$ with $12=2 \cdot 2 \cdot 3$ and $18=2 \cdot 3 \cdot 3 $ give
the Ulrich bundle
\[ \Sigma^{(b_1, \cdots,b_{12})} \cQ \otimes \Sigma^{(a_1, \cdots,a_{18})} \cS^{\vee} \]
with
\[{\footnotesize (b_1, \cdots,b_{12})=(187,186,177,176,119,118,109,108,51,50,41,40)} \]
and
\[{\footnotesize (a_1, \cdots,a_{18})=(40,40,38,38,36,36,22,22,20,20,18,18,4,4,2,2,0,0)} \]
meanwhile the sequences $(2,3,2)$ and $(2,3,3)$ with $12=2 \cdot 3 \cdot 2$ and $18=2 \cdot 3 \cdot 3 $ give the Ulrich
bundle
\[ \Sigma^{(b_1, \cdots,b_{12})} \cQ \otimes \Sigma^{(a_1, \cdots,a_{18})} \cS^{\vee} \]
with
\[{\footnotesize (b_1, \cdots,b_{12})=(187,186,177,176,167,166,85,84,75,74,65,64)} \]
and
\[{\footnotesize (a_1, \cdots,a_{18})=(64,64,62,62,60,60,34,34,32,32,30,30,4,4,2,2,0,0)}. \]
This example let us to stress that, in fact, the same integers placed in different order inside the sequence give
different Ulrich bundles.
\end{example}

In some particular cases, we can exactly give the number of $\GL(V)$-invariant Ulrich bundles we have. Namely,

\vspace{3mm}
\begin{corollary}\label{noUlrich}
Let $\Gr(k,n) \hookrightarrow \PP(\bigwedge^{k+1} V)$ be the Grassmannian embedded by the Pl\"{u}cker embedding.
\begin{itemize}
\item[(a)] If $k+1$ is a prime number, then there are exactly $ \phi(n-k)$ initialized $\GL(V)$-invariant Ulrich
    bundles being $\phi(n-k)$ the number of divisors of $n-k$.
     \item[(b)]If $n-k$ is a prime number, then there are exactly $ \phi(k+1)$ initialized $\GL(V)$-invariant
         Ulrich bundles being $\phi(k+1)$ the number of divisors of $k+1$.
     \end{itemize}
\end{corollary}
 \begin{proof} $(a)$ It follows from the fact that if $k+1$ is a prime number, we only have two ordered sequences $(k_1, \dots, k_s)$ such that $k+1=k_1 \cdots k_s$, namely  $(k_1)=(k+1)$ of length 1  and $(k_1,k_2)=(k+1,1)$ of length 2. Hence the only sequences we can consider
 verifying $n-k=n_1 \cdots n_s $ are $(n_1)=(n-k)$ of length 1 and $(n_1,n_2)=(a_i,b_i)$ of length 2 for any divisor $a_i$, $1 \leq a_i
 <n-k$ of
 $n-k$.  Hence, by Theorem \ref{main}, there are exactly $ \phi(n-k)$ initialized $\GL(V)$-invariant Ulrich bundles
 being $\phi(n-k)$ the number of divisors of $n-k$. The proof of $(b)$ is completely analogous.
 \end{proof}

\vspace{3mm}

 As a consequence of our main result Theorem \ref{main}, we are able to solve Problem \ref{problem} (b) for $\GL(V)$-invariant Ulrich bundles on $\Gr(k,n)$.

 \begin{corollary}\label{minimrankUlrich}
Let $\Gr(k,n) \hookrightarrow \PP(\bigwedge^{k+1} V)$ be the Grassmannian embedded by the Pl\"{u}cker embedding.
Then, the smallest possible rank for a $\GL(V)$-invariant Ulrich bundle on $\Gr(k,n)$ is
$$\frac{\prod _{1\le i < j \le k+1}(j-i)(n-k)}{k!(k-1)!\cdots 2!}.$$
\end{corollary}
 \begin{proof} By Theorem \ref{main}, we must look for the  pair of ordered sequences $(k_1, \cdots,k_s)$ and $(n_1,
 \cdots,n_s)$ such that the corresponding Ulrich bundle $\cE^{(k_1, \cdots,k_s)}_{(n_1, \cdots,n_s)}$
 has the smallest possible rank. By Lemma \ref{dim}, this is the case for the sequences $(k_1)=(k+1)$, $(n_1)=(n-k)$  and the rank
 of $\cE^{(k+1)}_{(n-k)}$ is exactly
 $$\frac{\prod _{1\le i < j \le k+1}(j-i)(n-k)}{k!(k-1)!\cdots 2!}.$$
 \end{proof}

\vspace{3mm}

We end the paper with another interesting consequence of our main result:

\begin{corollary}
Any initialized Ulrich bundle on $\Gr(k,n)$ has slope  $$ \mu=\frac{k(n-k-1)}{2} \cdot d $$
being $d:=\de(\Gr(k,n))= \frac{((k+1)(n-k))! k!(k-1)! \cdots 2!}{n!(n-1)! \cdots (n-k)!}$.
\end{corollary}
\begin{proof} According to Proposition \ref{ulrichslope} (a), all Ulrich bundle on $\Gr(k,n)$ has the same slope.
By Theorem \ref{main},  $\cE^{(k+1)}_{(n-k)}=\Sigma^{(b_1,\cdots,b_{k})}Q$
being $b_i=(k-i+1)(n-k-1)$ for $1 \leq i \leq k$ is an Ulrich bundle. Hence, it is enough to compute the slope of
$\cE^{(k+1)}_{(n-k)}$. By Lemma \ref{slope}, the slope of $\cE^{(k+1)}_{(n-k)}$ turns out to be
\[ \mu(\cE^{(k+1)}_{(n-k)})=\mu(\Sigma^{(b_1,\cdots,b_{k})}Q)= \frac{\sum_{i=1}^{k+1}(k-i+1)(n-k-1)}{k+1}\cdot d= \frac{k(n-k-1)}{2} \cdot d.\]
\end{proof}

Notice that,  in spite of the fact that the slope of an Ulrich bundle on  $\Gr(k,n)$
is determined, the uniqueness  of initialized Ulrich bundles is false as it has been shown in Theorem \ref{main}.

\end{document}